\documentclass[aip,reprint]{revtex4-1}

\usepackage{graphicx, enumerate, color,caption,subcaption}
\usepackage{natbib}

\usepackage{amsmath,amssymb,amsthm, amsfonts,bm,empheq}
\usepackage{cases}

\usepackage{comment,soul,xcolor,todonotes}
\usepackage[normalem]{ulem}
\usepackage[english]{babel}

\usepackage{hyperref}

\newtheorem{theorem}{Theorem}

\newtheorem{remark}{Remark}

\newcommand{\norm}[1]{\left\Vert#1\right\Vert}
\newcommand{\inpd}[2]{\left\langle #1, #2 \right\rangle}

\newcommand{\set}[1]{\left\{#1\right\}}

\newcommand{\Xb}{{\bar{X}}}

\newcommand{\C}{C}

\newcommand{\Real}{\mathbb {R}}

\newcommand{\tr}{\textsf{T}}

\newcommand{\eps}{\varepsilon}

\newcommand{\FW} {Freidlin-Wentzell}

\def\biz{\begin{itemize} }
\def\eiz{\end{itemize}}

\def\bena{\begin{enumerate}[<+-| alert@+>]}
\def\ben{\begin{enumerate}}
\def\een{\end{enumerate}}

\usepackage{color}

 \newcommand{\mJ}{\mathbb{J}}
 \newcommand{\N}{\mathbb{N}}
 \newcommand{\M}{\mathbb{M}}

\begin{document}
 
\title {Simplified Gentlest Ascent Dynamics for
Saddle Points in Non-gradient Systems }

\author{Shuting Gu}
\email{shutinggu2-c@my.cityu.edu.hk}
\affiliation{Department of Mathematics, City University of Hong Kong, 
Tat Chee Ave, Kowloon, Hong Kong SAR
}

\author{Xiang Zhou}
\email{xiang.zhou@city.edu.hk}
\thanks{Corresponding author.
The research of XZ  was supported by the grants from the Research Grants Council of the
Hong Kong Special Administrative Region, China (Project No. CityU 11304715 and 11337216).  }%
\affiliation{Department of Mathematics, City University of Hong Kong, Tat Chee Ave, Kowloon, Hong Kong SAR}


\date{\today}

\begin{abstract} The gentlest ascent dynamics (GAD) 
 ({\it Nonlinearity}, vol. 24, no. 6, p1831, 2011)  is a continuous time  dynamics
 coupling both the position and the direction variables to efficiently locate the
 saddle point with a given index. 
 These  saddle points play important roles
 in the   activated process of the randomly perturbed 
 dynamical systems.
 For index-1 saddle points in non-gradient systems,  the GAD
 requires two direction variables to approximate the  eigenvectors of the Jacobian matrix and its transpose, respectively,
 while in the gradient systems, these two directions collapse to be the single min mode of the Hessian matrix.
 In this note,  we present a simplified GAD which only  needs one direction variable  even for non-gradient systems.
 This new method not only  reduces computational cost for directions by half,
 but also  can avoid inconvenient operations on  the transpose of Jacobian matrix.
  We prove  the same convergence property for the simplified GAD
   as for the original GAD.  The motivation of our simplified GAD is
 its formal analogy  to the Hamiltonian dynamics  governing the exit dynamics
when the system is perturbed by small  noise.
Several non-gradient examples are presented to demonstrate our method,
including the two dimensional models and
the Allen-Cahn equation in the presence of shear flow.

\end{abstract}

\pacs{05.40, 05.70.Ln,82.40.Bj }
\keywords{saddle point, rare event, non-gradient system}

\maketitle 


\section{Introduction}

 Locating the saddle points has been of broad interest in many areas of scientific applications,
 especially for the understanding the exit process leaving 
 from linearly stable states when a dynamical system is randomly perturbed.
 In computational chemistry\cite{Energylanscapes},  one of the most important objects on the potential energy surface
 is the transition state which  is  the  saddle point with index 1, {\it i.e.},
the  unstable manifold
 is exactly one dimensional.
Such transition states are the bottlenecks on the most probable transition paths
between different local wells that describe  the random hoppings on the potential surface.
  The steepest descent flow that minimizes   the potential energy
  gives arise to gradient dynamical systems.
  For such gradient systems,
   a large amount of  numerical methods have been
developed to locate their saddle points, such as the eigenvector following method\cite{Crippen1971}, the dimer method \cite{Dimer1999} and the gentlest ascent dynamics(GAD) \cite{GAD2011,SamantaGAD}, the iterative minimization algorithm \cite{IMF2014,IMA2015} and others\cite{ART1998,DuSIAM2012}.

While most of these methods were designed for the gradient systems,
there are few of them applicable to  the non-gradient systems, which 
arise from many models in biology and fluid dynamics\cite{KS-WZE2009,WanYE_NS2dMAM_2015}.
One prominent example\cite{AO1997,MH2008} is the phase filed model such as the Allen-Cahn equation associated with a double-well potential, but  subject to  the influence of  shear flow.
The extra forcing from the fluid  certainly makes the gradient system become  a non-gradient model.
The saddle points in such non-gradient systems are still of  great importance 
 since they may be also relevant  to the non-equilibrium process in the 
randomly perturbed dynamical systems  \cite{FW2012}.

Among many saddle search methods mentioned previously, only  the GAD\cite{GAD2011}    proposed  by one of the authors in this note 
is capable to  
address the saddle point in  general dynamical systems,
by using two eigenvectors and  oblique projection. This result
  extends the saddle point search method to the non-gradient systems.
 In this note, we present a new form  of the  gentlest ascent dynamics  associated with the  following non-gradient system
 \begin{equation}\label{non-grad}
  \dot{x} = b(x),
 \end{equation}
 where $b$ is a smooth vector field  in $\Real^d$.
We are  interested in the index-1 saddle point of the vector field $b$.
 To locate the index-1 saddle point in    equation \eqref{non-grad}, the original GAD \cite{GAD2011} evolves a position variable $x$ and two direction variables $v$ and $w$
 so that the linearly stable state states of this new dynamics are index-1 saddle point of $b$.
The dynamics of $v$ and $w$ in the GAD needs
the product of   the Jacobian matrix $Db(x)$ and its transpose $Db(x)^\tr$
with $v$ and $w$, respectively.
 The matrix-vector multiplication $Db(x)v=\lim_{h\to 0} (b(x+hv)-b(x))/h$ can be
 easily approximated  by the finite difference method. But the difficult comes from the calculation of the transposed term $Db(x)^\tr w$, which
 lacks the interpretation of the directional derivative of $b(x)$.
In our simplified GAD below, we shall show that it suffices to use
the dynamics of {\it  one} directional variable (either $v$ or $w$), without affecting the convergence property of the original GAD.

Despite the simple form of our result,
we find  an interesting connection between the simplified GAD
and the underlying Hamilton's equation describing  the  optimal transition path in the randomly perturbed equation:
\begin{equation}\label{SDE}
d X = b(X)dt + \sqrt{\eps}dW,
\end{equation} where $W$ is the standard Brownian motion and $\eps$ is a small constant.
Indeed, the study of rare events in the system \eqref{SDE} is the most important motivation  to study the saddle points of the vector field $b$.
By the Freidlin-Wentzell large deviation theory\cite{FW2012}, the most probable transition path is a minimizer of the Freidlin-Wentzell action functional and this path, as a function of time,  satisfies the Hamilton's  equation with the zero Hamiltonian. The position and momentum in the Hamilton's equation might be   thought as  the counterpart  of the position   and   direction   in the GAD.
This formal analogy is indeed our original inspiration to derive our simplified GAD.

 The rest of the paper is organized as follows. In Section \ref{sec_GAD_Ham}, we propose the simplified GAD for non-gradient systems
 after  a short review of the  GAD. Then
 we explore  the relation between the simplified GAD and the Hamilton's dynamics.
 In addition, we apply the simplified GAD to the multiscale model of non-gradient slow-fast systems.
 Section \ref{sec_numer_ex} is our numerical examples.
In particular, we study  the Allen-Cahn equation in the presence of shear flow
and investigate how the shear  rate  affects  the transitions states in this system.
 The conclusions and discussions are given in Section \ref{sec_conclu}.

\section{ Method}\label{sec_GAD_Ham}

\subsection{Review of Gentlest Ascent Dynamics (GAD)}
The GAD in \cite{GAD2011} for the flow $\dot{x}(t) = b(x)$ involves a position variable $x$ and two direction variables $v$ and $w$ as follows:
\begin{subequations}\label{GAD}
\begin{empheq}[left=\empheqlbrace]{align}
  \dot{x}(t) &= b(x) -2 \frac{\inpd{b(x)}{ w }}{\inpd{w}{v} }v,\label{GAD-x}\\
  \gamma  \dot{v}(t) &=  J(x)v - \alpha v, \label{GAD-v}\\
 \gamma \dot{w}(t)&=  J(x)^\tr w - \beta w,  \label{GAD-w}
\end{empheq}
\end{subequations}
where $J(x) = D b(x)$ is the Jacobian matrix $\left(D b\right)_{ij} \doteq
\frac{\partial b_i}{\partial x_j}$, which is generally asymmetric.
$\alpha$ and $\beta$ are the Lagrangian multipliers to impose
certain  normalization conditions for $v$ and $w$. For instance,
if the normalization condition is  $\inpd{v}{v}\equiv \inpd{w}{v}\equiv 1$, then   $\alpha=\inpd{v} {J(x)v}$ and $\beta=2\inpd{w}{J(x) v}-\alpha$.  Equation \eqref{GAD}
is a flow in $\Real^{3d}$.

As a special case, the GAD for a gradient system $\dot{x}(t) = -\nabla V(x)$ only involves $v$:
\begin{subequations}
\begin{center}
\begin{empheq}[left=\empheqlbrace]{align}
  \dot{x}(t) &= -\nabla V(x) + 2 \frac{\inpd{\nabla V(x)}{ v}}{\inpd{v}{v}} v,\label{GAD-g-x}\\
\gamma \dot{v}(t) & =   - \nabla^2 V(x)v + \inpd{v} {\nabla^2 V(x)v}v. \label{GAD-g-v}
 \end{empheq}
\end{center}
\label{GAD-g}
\end{subequations}
 $\gamma > 0 $ is the relaxation parameter. A large $\gamma$ means
 a fast relaxation for the direction variable
 $v(t)$ toward to the steady state.
 For a frozen $x$, this steady state is the min mode of the Hessian $\nabla^2 V(x)$:
 the eigenvector corresponding to the smallest eigenvalue of $\nabla^2 V(x)$.

One of the authors \cite{GAD2011} proves   that
the above GAD (the general form \eqref{GAD} and the gradient form \eqref{GAD-g}) has the property that its  stable critical point  corresponds to  an index-1 saddle point of the original dynamics, $\dot{x}=b(x)$ or  $ \dot{x}=-\nabla V(x)$.
Our simplified GAD has the exactly same property, which  will  be given  below in details.

In the   GAD \eqref{GAD} for non-gradient systems, both $J(x)v$ and $J(x)^\tr w$ in \eqref{GAD-v} and \eqref{GAD-w} must be calculated.
One can apply the finite difference scheme to compute the matrix-vector multiplication $J(x)v$. But this trick could not be applied to the term $J(x)^\tr w $.  It can only be obtained by a numerical transpose operation.
The matrix-vector multiplication $J(x)^\tr w $ may impose a   severe computational challenge
for large scale problems.

\subsection{ Simplified GAD }\label{simpGAD}
Our new GAD takes {\it one} of the following two forms
(not simultaneously):
\begin{subnumcases}{\label{HGADv}}
 \dot{x}  =  b(x)-2 \inpd{b(x)}{v(t)} v(t) / \norm{v(t)}^2,
 \\
   \dot{v}  =  J (x) v - \inpd{v}{J  v} v,
\end{subnumcases}
or
\begin{subnumcases}{\label{HGADw}}
 \dot{x}  =  b(x)-2 \inpd{b(x)}{w(t)} w(t) / \norm{w(t)}^2,
 \\
   \dot{w}  =  J^\tr(x) w - \inpd{w}{J^\tr w} w.\label{HGADw_w}
\end{subnumcases}
So, the simplified GAD is always a flow in $\Real^{2d}$.
Initially,
$\norm{v_0}=1$ or  $\norm{w_0}=1$ so that
$v$ and $w$ are always unit vectors.
The difference between \eqref{HGADv} and \eqref{HGADw} is the matrix-vector multiplication $J(x)v$ or $J(x)^\tr w$.
As discussed above, to avoid computing $J(x)^\tr w$,
one prefers the equation  \eqref{HGADv}  for the simplified GAD in practice.
It will be seen later that in theory,  equation \eqref{HGADw}
may be  of more interest.
For the gradient system $\dot{x}(t) = -\nabla V(x)$, $J=-H$,
where $H=\nabla^2 V=H^\tr$ is the Hessian matrix, the above two forms
are identical and  become the  GAD \eqref{GAD-g} for the gradient system.

\begin{remark}\label{rm1}
 A positive constant $\tau$ can be used  in the simplified  GAD: 
   $\dot{v}\to \tau \dot{v}$ ( or $\dot{w}\to \tau \dot{w}$ as in equation \eqref{GAD}) ,
 to represent the time scale ratio between $x$ and $v$ ( or $w$).
 We drop this factor to ease the presentation.
\end{remark}

The simplified GAD \eqref{HGADv} or \eqref{HGADw} converges to the index-1 saddle point of the original dynamics $\dot{x} = b(x)$; see the following theorem.
The proof is quite similar to that for the original GAD
\cite{GAD2011}.

\begin{theorem}\label{theorem1}

(a) If $(x_*, v_*)$ is a fixed point of the simplified GAD \eqref{HGADv}, and $v_*$ is the normalized vector, $\|v_*\| = 1$, then $v_*$ is the eigenvector of $J(x_*)$ corresponding to an eigenvalue $\lambda_*$, {\it i.e.},
$$ J(x_*)v_* = \lambda_* v_*, $$
and $x_*$ is a fixed point of the original dynamics system, {\it i.e.}, $b(x_*) = 0$;

(b) Let $x_s$ be a fixed point of the dynamical system $\dot{x} = b(x)$. If the Jacobian matrix $J(x_s)$ has $n$ distinct real eigenvalues $\lambda_1, \lambda_2, \cdots, \lambda_n$ corresponding to the $n$ linearly independent eigenvectors $v_i$, {\it i.e.}, $$J(x_s)v_i = \lambda_i v_i, i = 1,2,\cdots,n$$  and $\|v_i\| = 1, \forall i$. Then $(x_s,v_i), \forall i$, is a fixed point of the simplified GAD \eqref{HGADv}. Furthermore, there is one fixed point $(x_s,v_{i^\prime})$ among these $n$ fixed points, which is linearly stable if and only if $x_s$ is an index-1 saddle point of the original dynamical system $\dot{x}=b(x)$ and the eigenvalue $\lambda_{i^\prime}$ corresponding to $v_{i^\prime}$ is the only positive eigenvalue of $J(x_s)$.

\end{theorem}

\begin{proof}
(a) By the condition that $(x_*, v_*)$ is a fixed point of the simplified GAD \eqref{HGADv}, we have

\begin{subnumcases}{\label{cond}}
  b(x_*) - 2 \inpd{b(x_*)}{v_*} v_* = 0, \label{cond1}
 \\
 J (x_*) v_* = \inpd{v_*}{J(x_*)  v_*} v_* . \label{cond2}
\end{subnumcases}
Equation \eqref{cond2} implies that $v_*$ is the eigenvector of $J(x_*)$ corresponding to the eigenvalue $\lambda_* \doteq \inpd{v_*}{J(x_*)  v_*}.$

Making inner product with $v_*$ on both sides of \eqref{cond1}, we can get
$$ \inpd{b(x_*)}{v_*} - 2 \inpd{b(x_*)}{v_*} \inpd{v_*}{v_*} = 0 . $$
Since $\|v_*\|=1$, we have $\inpd{b(x_*)}{v_*} = 0$. Thus $b(x_*)=0$ by \eqref{cond1}.

(b) Since $x_s$ is a fixed point of the system $\dot{x} = b(x)$, we have $b(x_s)=0$, thus
\begin{equation}\label{eq_b}
b(x_s)-2 \inpd{b(x_s)}{v_i} v_i / \norm{v_i}^2 = 0, \quad i = 1,2,\cdots,n.
\end{equation}
Since $J(x_s)v_i = \lambda_i v_i$, by taking inner product with $v_i$ on both sides and using the condition $\|v_i\|=1$, we get $\lambda_i = \inpd{J(x_s) v_i}{v_i}$, and
\begin{equation}\label{eq_v}
J(x_s)v_i - \inpd{J(x_s) v_i}{v_i} v_i = 0, \quad i = 1,2,\cdots,n.
\end{equation}
Equation \eqref{eq_b} and \eqref{eq_v} imply that $(x_s,v_i)$ is the fixed point of the simplified GAD \eqref{HGADv} for all $i= 1,2,\cdots,n$.

Next, we write down the eigenvalues and corresponding eigenvectors of the Jacobian matrix of the simplified GAD at any fixed point $(x_s, v_i)$. First, the Jacobian matrix of the simplified GAD \eqref{HGADv} has the following expression:
\begin{equation}\label{J1_GAD}
\tilde{\mJ}(x_s,v_i)=\\
\begin{bmatrix}
\N_1:=J - 2 \lambda_i v_i v_i^\tr, &
0 \\
* , & \M_1:=J - \lambda_i - v_i v_i^\tr (\lambda_i+J)
\end{bmatrix}.
\end{equation}
The eigenvalues of $\tilde{\mJ}$ can be obtained from the eigenvalues of its two diagonal blocks $\N_1$ and $\M_1$. It can be verified that

\begin{align*}
\N_1 v_i & = J v_i - 2\lambda_i v_i v_i^\tr v_i = -\lambda_i v_i,\\
\N_1 v_j & = J v_j - 2\lambda_i v_i v_i^\tr v_j = \lambda_j v_j,\\
\M_1 v_i &= J v_i - \lambda_i v_i - v_i v_i^\tr (\lambda_i+J) v_i = -2 \lambda_i v_i,
\end{align*}

and
\[
\begin{split}
& \M_1 (v_j - (v_j^\tr v_i)v_i) = \M v_j - (v_j^\tr v_i)\M v_i\\
&= (J - \lambda_i - v_i v_i^\tr(\lambda_i + J) )v_j + 2\lambda_i (v_j^\tr v_i) v_i \\
&= (\lambda_j - \lambda_i ) v_j - v_i (\lambda_i + \lambda_j) v_i^\tr v_j + 2\lambda_i (v_j^\tr v_i) v_i\\
&= (\lambda_j - \lambda_i ) v_j - v_i (\lambda_j - \lambda_i) v_i^\tr v_j \\
&= (\lambda_j - \lambda_i ) (v_j - (v_i^\tr v_j)v_i ).
\end{split}
\]
 Hence the eigenvalues of the Jacobian matrix $\tilde{\mJ}$ at any fixed points $(x_s,v_i), i=1,2,\cdots,n$ are
\begin{equation}\label{eigvalue}
-2\lambda_i, -\lambda_i, \{\lambda_j: j\neq i \}, \{\lambda_j-\lambda_i: j\neq i \}.
\end{equation}
The linear stability condition is that all the above eigenvalues of $\tilde{\mJ}$ are negative. Thus one fixed point $(x_s, v_{i^\prime})$ is linearly stable if and only if $\lambda_{i^\prime} > 0 $ and all other eigenvalues $\lambda_j < 0$ for $j \neq i^\prime$. In this case, the fixed point $x_s$ is an index-1 saddle point of the system $\dot{x} = b(x)$.

\end{proof}

\begin{remark}\label{rm2}
Theorem \ref{theorem1} also holds for the simplified GAD \eqref{HGADw}. In this case,  the Jacobian matrix of the simplified GAD \eqref{HGADw} becomes
\begin{equation}\label{J2_GAD}
\tilde{\mJ}(x_s,w_i)=
\begin{bmatrix}
\N_2 :=J - 2 \lambda_i w_i w_i^\tr, &
0 \\
* , & \M_2 :=J^\tr - \lambda_i - w_i w_i^\tr (\lambda_i + J^\tr )
\end{bmatrix}
\end{equation}
with the same eigenvalues \eqref{eigvalue} as the Jacobian matrix of the simplified GAD \eqref{HGADv}.

\end{remark}

\subsection{Relation with Hamilton's  equation}
In this part, we discuss the Hamilton's equation
associated with the rare event study of the equation \eqref{SDE}.
According to the \FW\ large deviation principle (LDP)   \cite{FW2012}, as the noise amplitude $\eps$ in equation \eqref{SDE} tends to zero, the most probable transition path over the time interval $[0, T]$  of the system \eqref{SDE} is the minimizer of the
following \FW\ action functional \begin{equation}
S[\phi] = \int_0^T L(\phi,\dot{\phi}) dt,
\end{equation}
where the Lagrangian $L(x,y)$ is defined as
\begin{equation}
L(x,y) := \frac{1}{2} \inpd{y-b(x)}{y-b(x)}.
\end{equation}
$\inpd{\cdot}{\cdot}$ is the inner product in $\Real^d$.
The Hamiltonian $H(x,p)$, as the conjugate of   the Lagrangian $L(x,y)$, is
\begin{equation}
H(x,p) = \inpd{b(x)}{p} + \inpd{p}{p}/2.
\end{equation}
It is well-known that the minimizer of $S[\phi]$,
denoted as $x(t)$, satisfies the Hamilton's  equations
 \begin{subnumcases}{\label{HODE}}
 \dot{x}  =  H_p=b(x) + p(t),\label{HODE-x}
 \\
   \dot{p}  = -H_x=- J(x)^\tr p(t),
\end{subnumcases}
where  $p(t)$ is the (generalized) momentum.
 $J(x)= D b(x)$ is the Jacobian matrix we have used before in the GAD.
The eigenvalues of $J(x)$ are denoted as $\set{\lambda_i}$.
Equation \eqref{HODE} looks  superficially  similar to equation \eqref{HGADw} with 
two   differences:  (i) the signs before $J(x)^\tr$
are the opposite and
(ii) the momentum $p$ is not a unit vector
as the direction variable $w$.
In fact, the critical point of \eqref{HODE} is $(x_*,p_*)$
where $b(x_*)=0$ and $p_*=0$ by assuming that $J(x_*)$ is non-degenerate.
Assume $J(x)$ has the right-eigenvectors $v_i$ and the left-eigenvectors $w_i$:
 $$Jv_i=\lambda_i v_i, \mbox{ and }~   J^\tr w_i = \lambda_i w_i, ~~ 1\leq i\leq d,$$
where all eigenvalues are assumed distinctive
and the left or right eigenvectors both form a basis of $\Real^d$.
We introduce the normalized unit vector $u$ to represent the direction of $p$.
Define the scalar $l\doteq \norm{p}^2$, then $u=p/\sqrt{l}$ and
$\dot{l}=2\inpd{p}{\dot{p}}=-2\inpd{p}{J^\tr p} = -2l \inpd{u}{J^\tr u}$.
So,
\begin{equation}\label{166}
\dot{u}=\frac{d}{dt}\left(\frac{p}{\sqrt{l}}\right)= -J^\tr (x)u +\inpd{u}{J^\tr u} u.
\end{equation}
By the important  zero-Hamiltonian condition $H\equiv 0$ (\cite{FW2012}), we have
$$l=\norm{p}^2 = -2\inpd{b}{p}= -2\sqrt{l} \inpd{b}{u};$$
that is,
$$ l=0, \quad \text{or} \quad \sqrt{l}=-2\inpd{b(x)}{u}. $$
$l=0$ means $p=0$, which corresponds to the original dynamics $\dot{x}=b(x)$.
$l$ is not always zero for the exit dynamics,
then $\sqrt{l}=-2\inpd{b(x)}{u}$ and the equation \eqref{HODE-x} becomes
\begin{equation}\label{168}
 \dot{x} = b(x) + \sqrt{l} u= b(x)-2 \inpd{b(x)}{u} u.
 \end{equation}

So far, by \eqref{166} and \eqref{168}, we get the momentum-normalized  version for the Hamilton's equation \eqref{HODE} restricted on the zero-$H$ hypersurface:
\begin{subnumcases}{\label{HODEn}}
 \dot{x}  =  b(x)-2 \inpd{b(x)}{u(t)} u(t) / \norm{u(t)}^2,
 \label{HODEn_x} \\
   \dot{u}  =  -J^\tr(x) u + \inpd{u}{J^\tr u} u.\label{HODEn_w}
\end{subnumcases}
$\norm{u_0}=1$ is assumed.
Note that this dynamics \eqref{HODEn} is
not  exactly identical  to  the original Hamilton's equation \eqref{HODE}
since the branch of $p\equiv 0$ has been discarded.

Now, the only difference between the Hamilton's
equation \eqref{HODEn} and the simplified GAD \eqref{HGADw}
is the opposite sign on the right hand sides of  \eqref{HODEn_w}
and  \eqref{HGADw_w}. 
By Remark \ref{rm2}, the Jacobian matrix of  \eqref{HODEn}
is $\begin{bmatrix} \N_2, & 0 \\ *, & -\M_2\end{bmatrix}$, whose eigenvalues are
$
-\lambda_i , 2\lambda_i , \set{\lambda_j, j\neq i} , \set{\lambda_i -\lambda_j,  j\neq i}.
$
The position dynamics in \eqref{HGADw} and \eqref{HODEn}    have the same form
of applying the projection matrix $I - 2 ww^\tr $ or $I-2uu^\tr$ in front of the original force $b(x)$.
The difference is which direction they select.
If  $x$  were frozen,  the $w$ dynamics in equation \eqref{HGADw_w}
picks up the
  least   stable direction
  while the Hamilton equation's  momentum direction uses
the most stable direction.
Thus the GAD \eqref{HGADw} can converge to  the saddle point of the vector field $b(x)$ while
the Hamiltonian dynamics  \eqref{HODEn} has no stable steady state.
So one may view the simplified GAD
as a modification of the Hamilton's equation by  flipping the sign of  the (normalized)
momentum to stablized the saddle point.
 Note that  although we can introduce a factor $\gamma$ for \eqref{HGADw_w} as shown in Remark \ref{rm1}
 to speed up the clock for the direction dynamics,
 there is no  such a freedom for the Hamilton's equation \eqref{HODEn_w}.

\subsection{ Application  to multiscale model}
The GAD was extended to the slow-fast stochastic system in \cite{MsGAD2017}
and the resulted method is called MsGAD.
As a corollary of our result,
 the simplified GAD here can  be   applied to this multi-scale model straightforwardly.
 For the backgrounds and more details, the reader can refer to
\cite{MsGAD2017}.
We here directly present  the scheme based on the
above simplified GAD.
 The slow-fast system in consideration is
\begin{subequations}{ \label{XY}}
 \begin{empheq}[left={ \empheqlbrace\,}]{align}
\dot {X}^\eps(t) & = f(X^\eps,Y^\eps) , \label{X}    \\
 \dot {Y}^\eps(t) &=  \frac{1}{\eps}b (X^\eps,Y^\eps)   + \frac{1}{\sqrt{\eps}} \sigma(X^\eps,Y^\eps) \eta(t),  \label{Y}
     \end{empheq}
 \end{subequations}
 where    $\eps$ is a small parameter
 and $\eta$ is the noisy perturbation.
 $X^\eps$ is {the} slow variable and $Y^\eps$ is {the} fast variable. When $\eps$ goes to zero, the effective dynamics of the slow variable is
 \begin{equation} \label{Xbar}
 \dot{\Xb} = F(\Xb), ~~\mbox{ where }  F(x) \doteq \int f(x,y) \mu_x(dy),
 \end{equation}
 where  $\mu_x(dy)$ is the invariant measure of the fast process  with the density function denoted by $\rho(x,y)$.
  $F$ usually does not have analytical form. The simplified multiscale GAD
 for the saddle point of equation \eqref{Xbar} is
\begin{subequations}\label{MsGAD1}
\begin{center}
\begin{empheq}[left=\empheqlbrace]{align}
  \dot{x}^\eps(t) &= f(x^\eps,y^\eps) -2 \frac{\inpd{f(x^\eps,y^\eps)}{ v^\eps }}{\inpd{v^\eps}{v^\eps}} v^\eps,\label{MsGAD1-x}\\
   \dot {y}^\eps(t) &=  \frac{1}{\eps} b (x^\eps,y^\eps)   + \frac{\sigma(x^\eps,y^\eps)}{\sqrt{\eps}} \eta(t) ,\label{MsGAD1-y}\\
 \dot{v}^\eps(t) &= \left(D_x f (x^\eps,y^\eps)  + \C(x^\eps,y^\eps) \right)v^\eps
  - \alpha^\eps v^\eps, \label{MsGAD1-v}
\end{empheq}
\end{center}
\label{GAD1}
\end{subequations}
or

\begin{subequations}\label{MsGAD2}
\begin{center}
\begin{empheq}[left=\empheqlbrace]{align}
  \dot{x}^\eps(t) &= f(x^\eps,y^\eps) -2 \frac{\inpd{f(x^\eps,y^\eps)}{ w^\eps }}{\inpd{w^\eps}{w^\eps}} w^\eps,\label{MsGAD2-x}\\
   \dot {y}^\eps(t) &=  \frac{1}{\eps} b (x^\eps,y^\eps)   + \frac{\sigma(x^\eps,y^\eps)}{\sqrt{\eps}} \eta(t) ,\label{MsGAD2-y}\\
 \dot{w}^\eps(t)&=\left( D_x f (x^\eps,y^\eps) +   \C(x^\eps,y^\eps)  \right )^\tr  w^\eps -\beta^\eps w^\eps, \label{MsGAD2-w}
\end{empheq}
\end{center}
\label{GAD1}
\end{subequations}
where $D_x f(x,y)$ is the Jacobian matrix of $f(x,y)$ with respect to $x$. $\alpha = \inpd{v}{(D_x f + C )v}, \beta = \inpd{w}{(D_x f + C)^T w}$,
 $C(x,y) =  (f(x,y) -F(x))   \otimes ( g(x,y) - G(x))$,  $g(x,y)=-\nabla_x U(x,y), U(x,y)=-\log \rho(x,y)$ and $G(x)=\int g(x,y)\mu_x(dy)$.
 
\section{Numerical examples}\label{sec_numer_ex}

\subsection{A two-dimensional deterministic system}
The first test is to find the saddle point of the following two dimensional non-gradient system 

\begin{equation}\label{eg1-ave}
\dot{x}_i = - \sum_{j=1}D_{ij}x_j + \frac{\sigma^2}{2}\Gamma_i(x), \quad i=1,2,
\end{equation}
where
$\sigma^2=10$,
$D=\begin{bmatrix} 0.8 & -0.3  \\ -0.2 & 0.5 \end{bmatrix}
$ and $\Gamma_i(x)= \left( 1+ (x_i-5)^2 \right)^{-1},~ i = 1,2.$
 This dynamics has two stable fixed points $m_1 = (0.5931, 0.7655), m_2 = (5.8770, 6.2507)$ and
   a unique saddle point  $s = (1.7954, 3.3088) $.
   Figure \ref{traj_HGAD}  shows the simplified GAD trajectories of the $x$ component (solid lines) starting from $m_1$ and $m_2$ respectively.

\begin{figure}[!htb]
\begin{center}
  \includegraphics[scale=0.36]{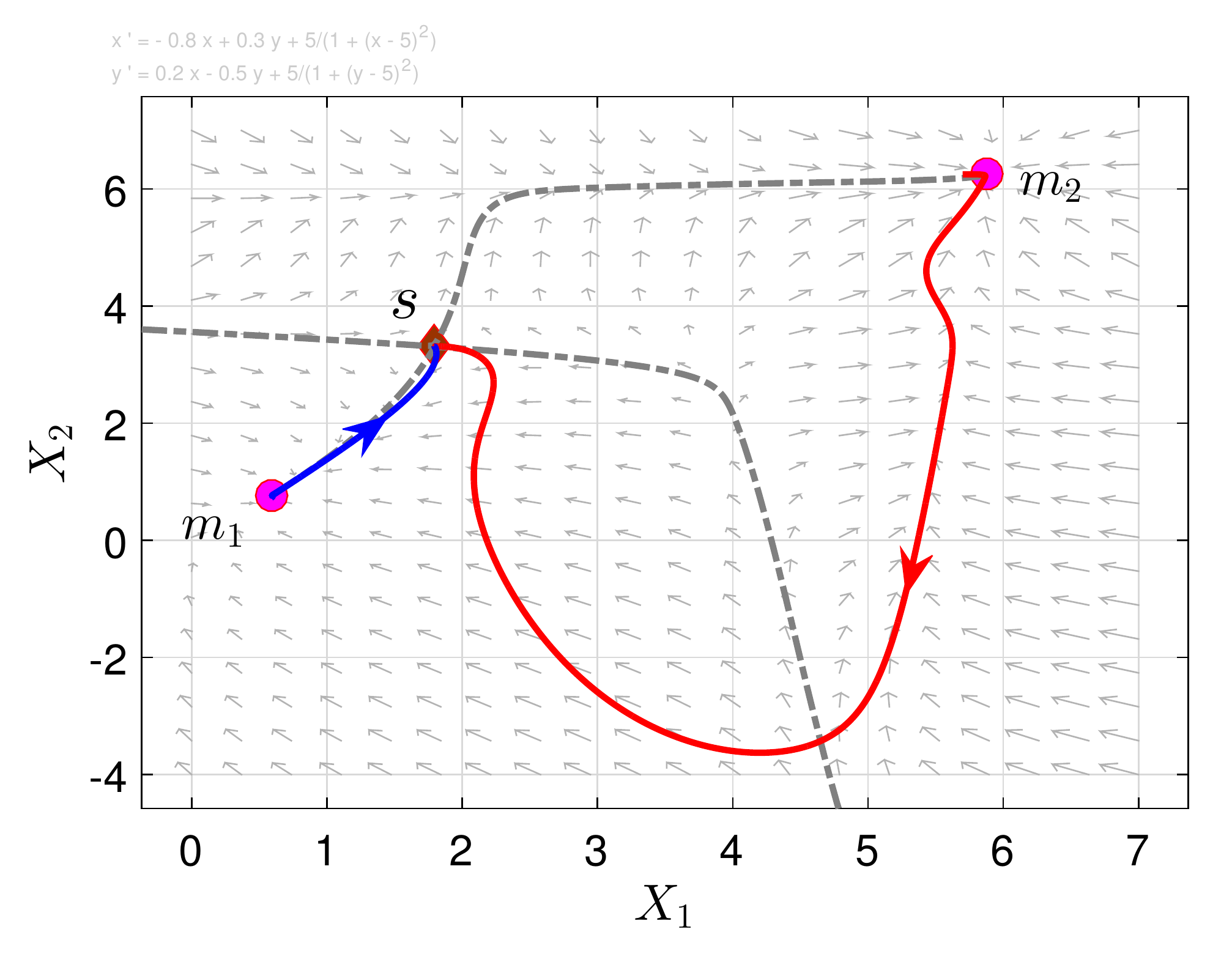}
\end{center}
 \caption{GAD trajectories of the $x$ component from two locally stable fixed points
 ($m_1$ and $m_2$) to the saddle point $s$.
The flow indicated by the  arrows corresponds to  the non-gradient  system \eqref{eg1-ave}.
The dash-dotted curves are the stable/unstable manifolds of the saddle point;
they determine  the basin boundaries of the two stable fixed points.
The blue and red curves with arrows   are the trajectories of the $x$ component for the simplified GAD applied to
the dynamics  \eqref{eg1-ave} with the initial vector $v=[1,0]$ and $[0,1]$, respectively.
}
\label{traj_HGAD}
 \end{figure}

\subsection{A two-dimensional slow-fast system}
Consider a slow-fast system in \cite{MsGAD2017},
\begin{subequations}
\begin{center}
\begin{empheq}[left=\empheqlbrace]{align}
  \dot{x}_i   &= - \sum_{j}D_{ij}x_j + y_i^2,  \label{eg1-x}\\
 \dot{y}_i & =   - \frac{1}{\eps}\frac{ y_i}{ \Gamma_i(x)}  + \frac{1}{\sqrt{\eps}}\sigma \eta(t),\label{eg1-y}
 \end{empheq}
\end{center}
\label{eg1}
\end{subequations}
where
$D=\begin{bmatrix} 0.8 & -0.2  \\ -0.2 & 0.5 \end{bmatrix}
$, which is different from the $D$ matrix in the first example
\eqref{eg1-ave}. $\sigma^2$ and  $\Gamma_i(x)$ are the same as in the first example. $\eta$ is the standard white noise.
 We are interested in the saddle point of the   effective dynamics which is the limit of \eqref{eg1} as $\eps\to0$.
 For this special case, it happens to have
 the following closed form for the effective dynamics
\begin{equation}\label{eg2-ave}
\dot{\Xb}_i = - \sum_{j}D_{ij}\Xb_j + \frac{\sigma^2}{2}\Gamma_i(\Xb).
\end{equation}
 Equation \eqref{eg2-ave} has three
 stable  fixed points $m_1 = (0.4643, 0.6985), m_2 = (2.2038, 5.9804)$ and $m_3 = (5.7109, 6.2369)$
 as well as  two saddle points   $s_1 = (1.2842, 3.4484), s_2 = (3.5689, 6.0735)$. Refer to  Figure \ref{traj_Ms}.
 To test our method, we use
  the  heterogeneous multiscale method(HMM)  to solve the simplified MsGAD \eqref{MsGAD1}
   numerically,
 without using any information of the
 analytical form in equation \eqref{eg2-ave}.
  Figure \ref{traj_Ms} shows four GAD trajectories of the $x$ component (black solid lines) with  different initial values.

 \begin{figure}[!htb]
\begin{center}
\includegraphics[scale=0.36]{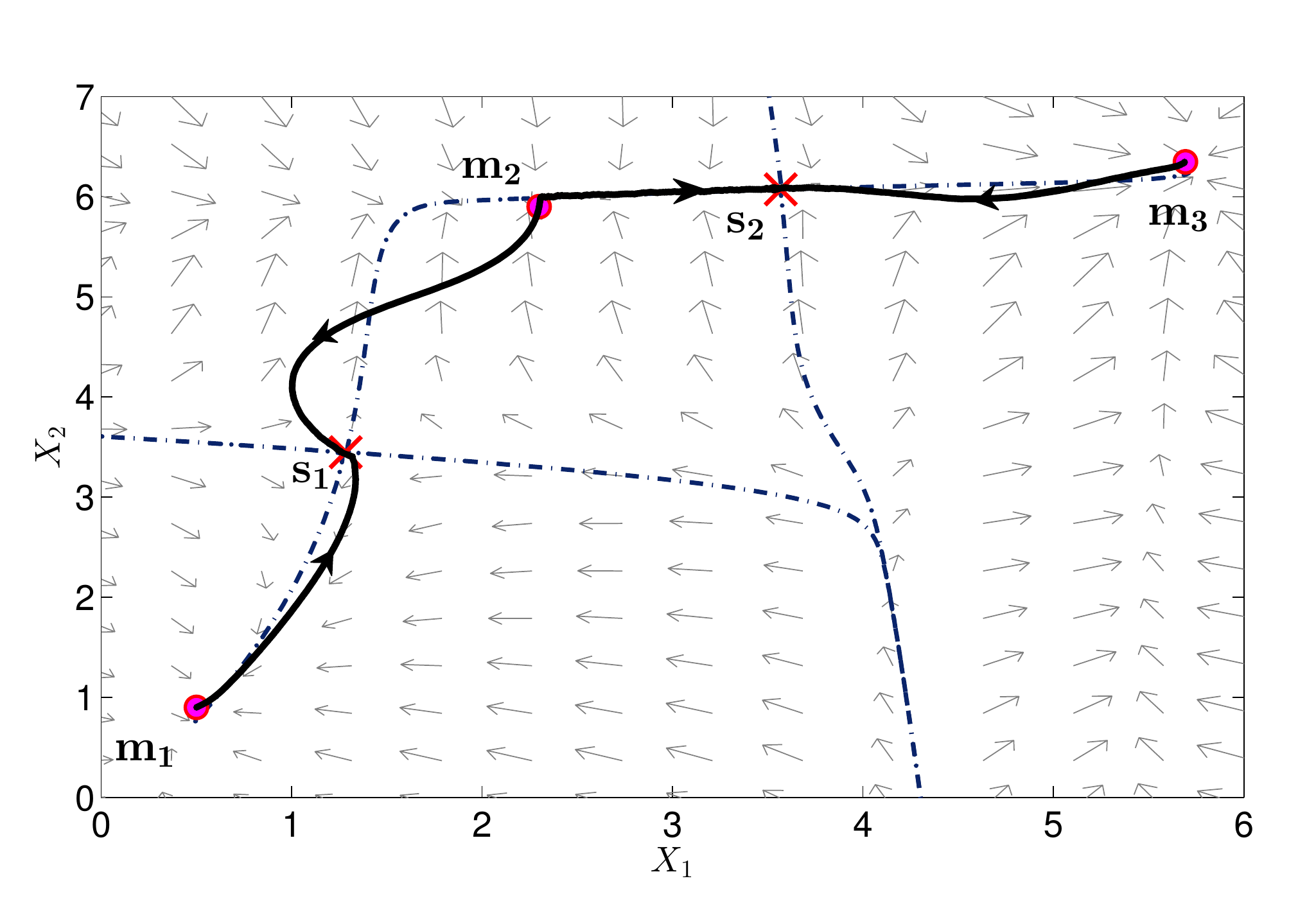}
\end{center}
 \caption{GAD trajectories from three
 stable fixed points
 ($m_1, m_2$ and $m_3$) to two different saddle points ($s_1$ and $s_2$).
The flow indicated by the  arrows corresponds to  the
effective  dynamics \eqref{eg2-ave}.
The dash-dotted curves are the stable/unstable manifolds of the two saddle points under the flow \eqref{eg2-ave}.
The black solid curves with arrows marked are the trajectories of the simplified MsGAD by the HMM.}
\label{traj_Ms}
 \end{figure}

\subsection{Nucleation in the presence of shear  flow}
As the last example, we consider   a more challenging problem: the nucleation in the
reaction-diffusion equation in the presence of shear. Nucleation   is a very important physical phenomenon \cite{MH2008,AO1997, LLinProj2010, AR1998, LLQ2007} and the nucleus  is usually described by the saddle point of the Ginzburg-Landau free energy. In
 the case of gradient systems
 purely driven by the free energy, the string method \cite{String2002,String2007} can be applied to calculate the minimum energy path.
 When the shear flow field
 is in presence,
 one is faced with a non-gradient systems
 and in principle, one needs the minimum action method
 \cite{MAM2004} to compute the minimum action path and the minimal action \cite{MH2008}. The saddle point can be extracted  after the whole path is computed.
  By our new method, however, the saddle point in the shear flow case can be calculated directly.
The Ginzburg-Landau free energy
 of the order parameter  $\phi(x,y)$ is 
\begin{equation}
E(\phi) = \int_\Omega   \frac{\kappa}{2} |\nabla\phi|^2 + \frac{1}{4}(1-\phi^2)^2 ~dx dy,
\end{equation}
where $\kappa=0.01$,  the domain $\Omega = [0,1]\times [0,1] $. The periodic boundary condition is considered.
We study two cases of the shear flow as  illustrated in
 Figure \ref{fig:2D_shear_flow}, then
 the corresponding  dynamics of the Allen-Cahn equation in the presence of the shears  are
 \begin{equation}\label{GL_dyna}
\partial_t \phi = -\frac{\delta E}{\delta\phi} + \gamma \sin(2\pi y) \partial_x \phi,
\end{equation}
and
\begin{equation}\label{GL_dyna_new}
\partial_t \phi = -\frac{\delta E}{\delta\phi} + \gamma \sin(2\pi y) \partial_x \phi + \gamma \sin(2\pi x)\partial_y \phi,
\end{equation}
respectively, where $\gamma$ is the shear rate and
the Fr\'echlet derivative $ \delta_\phi E = -\kappa \Delta \phi - \phi + \phi^3. $

\begin{figure}[htbp]
\begin{center}
\begin{subfigure}[b]{0.2\textwidth}
\includegraphics[width=\textwidth]{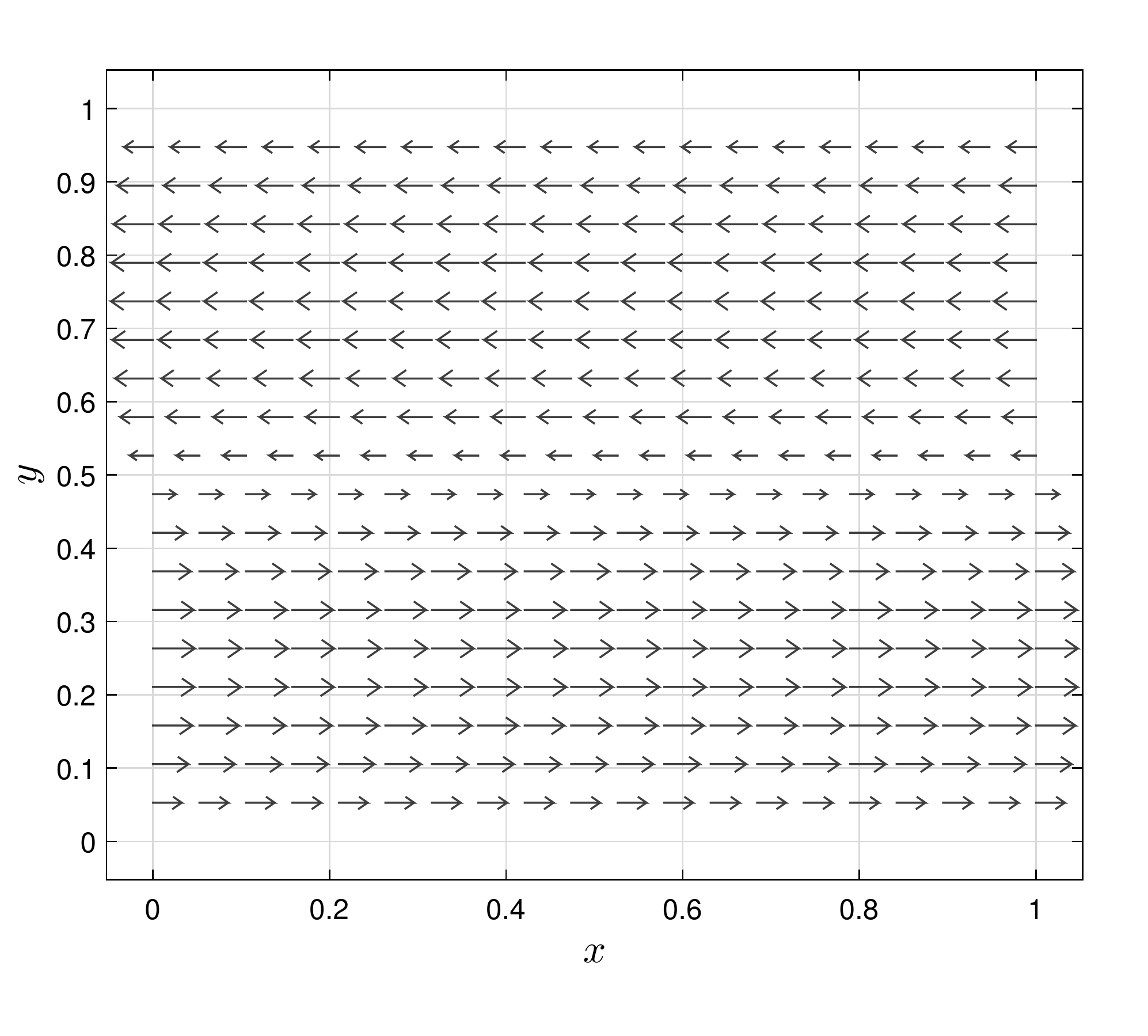}
 \label{fig:shear1}
\end{subfigure}
\begin{subfigure}[b]{0.2\textwidth}
\includegraphics[width=\textwidth]{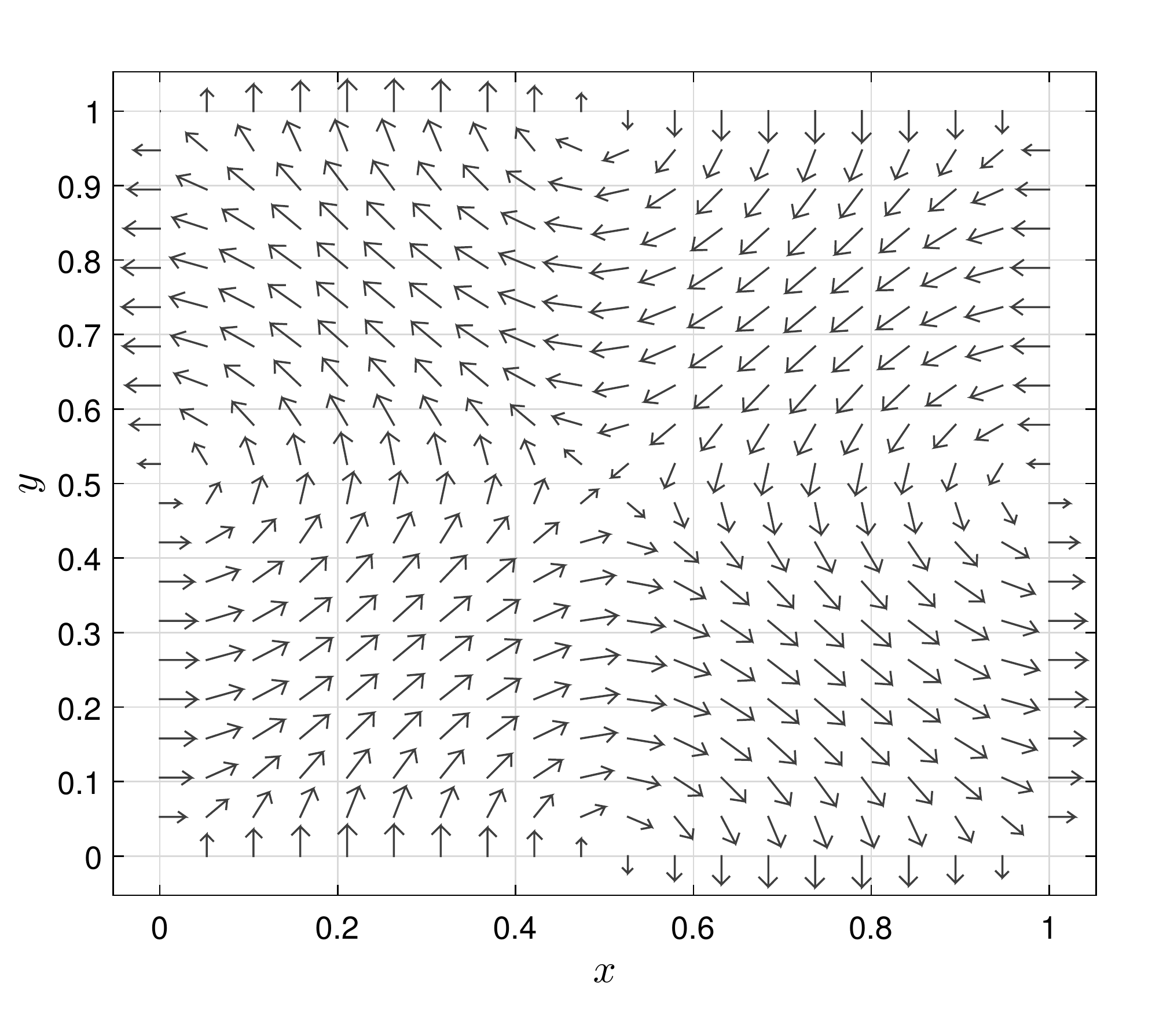}
 \label{fig:shear2}
\end{subfigure}
\caption{Vector fields of the two types of  shear flows.
}
\label{fig:2D_shear_flow}
\end{center}
\end{figure}

We want to locate the index-1 saddle point in the dynamics \eqref{GL_dyna} and \eqref{GL_dyna_new} by the simplified GAD in Section \ref{simpGAD}.
Denote the right hand side in \eqref{GL_dyna} or \eqref{GL_dyna_new} by $b(\phi)$, the simplified GAD \eqref{HGADv} in this case is
\begin{subnumcases}{\label{GL_GAD}}
 \partial_t \phi  =  b(\phi)-2 \inpd{ b(\phi)}{v} v / \norm{v}^2,\label{GL_GAD_phi}
 \\
   \partial_t v  =  D b (\phi) v - \inpd{v}{(D b)  v} v/\norm{v}^2,\label{GL_GAD_v}
\end{subnumcases}
  where $v=v(x,t)$ and $\inpd{\cdot}{\cdot}$, $\norm{\cdot}$ is the $L^2$ inner product and norm in space.

\begin{remark}
Here the dynamics is a PDE model
and we can have the analytical expression for the
Jacobian and its transpose.  We take the case in equation
\eqref{GL_dyna}  as an example to show $Db$
and its adjoint $(Db)^\tr$.
$b(\phi)=\kappa \Delta \phi + \phi - \phi^3 + \gamma \sin(2\pi y)\partial_x \phi$.
$Db(\phi) v =\kappa \Delta v   + v - 3 \phi^2 v  +\gamma \sin(2\pi y) \partial_x v $.
Then
$(Db(\phi))^\tr w =\kappa \Delta w   + w - 3 \phi^2 w  -\gamma \sin(2\pi y) \partial_x w $
since the adjoint of $\partial_x$ is $-\partial_x$.
This example shows that when $b$ is a differential operator,
one may obtain the ``transpose''  (adjoint) of the Jacobian   analytically.
Then the two forms of the simplified GAD \eqref{HGADv} and \eqref{HGADw} are both applicable in such cases.
\end{remark}

 In the numerical test, we use the mesh point $N_x = N_y = 128$
 in the finite difference method
 for spatial discretization.
   The two metastable states are always $\phi \equiv 1$ and $\phi\equiv -1$ regardless of the shear flow.  By solving the simplified GAD \eqref{GL_GAD}, we get different
   index-1 saddle points  for various $\gamma$.
     We are interested in 
the impact of shear flow
  on the profiles of the saddle point.
It is noted that  the steady states for any shear 
preserve the symmetry $\phi \to -\phi$
and 
 equation \eqref{GL_dyna_new}
additionally preserves the second symmetry  
$\phi(x,y)\to \phi(y,x)$. So there are multiple
symmetric  images 
for the same steady states.
All of our plots below  show  only one of the symmetric images. 

\begin{figure}[htbp]
\begin{center}
\begin{subfigure}[b]{0.18\textwidth}
\includegraphics[width=\textwidth]{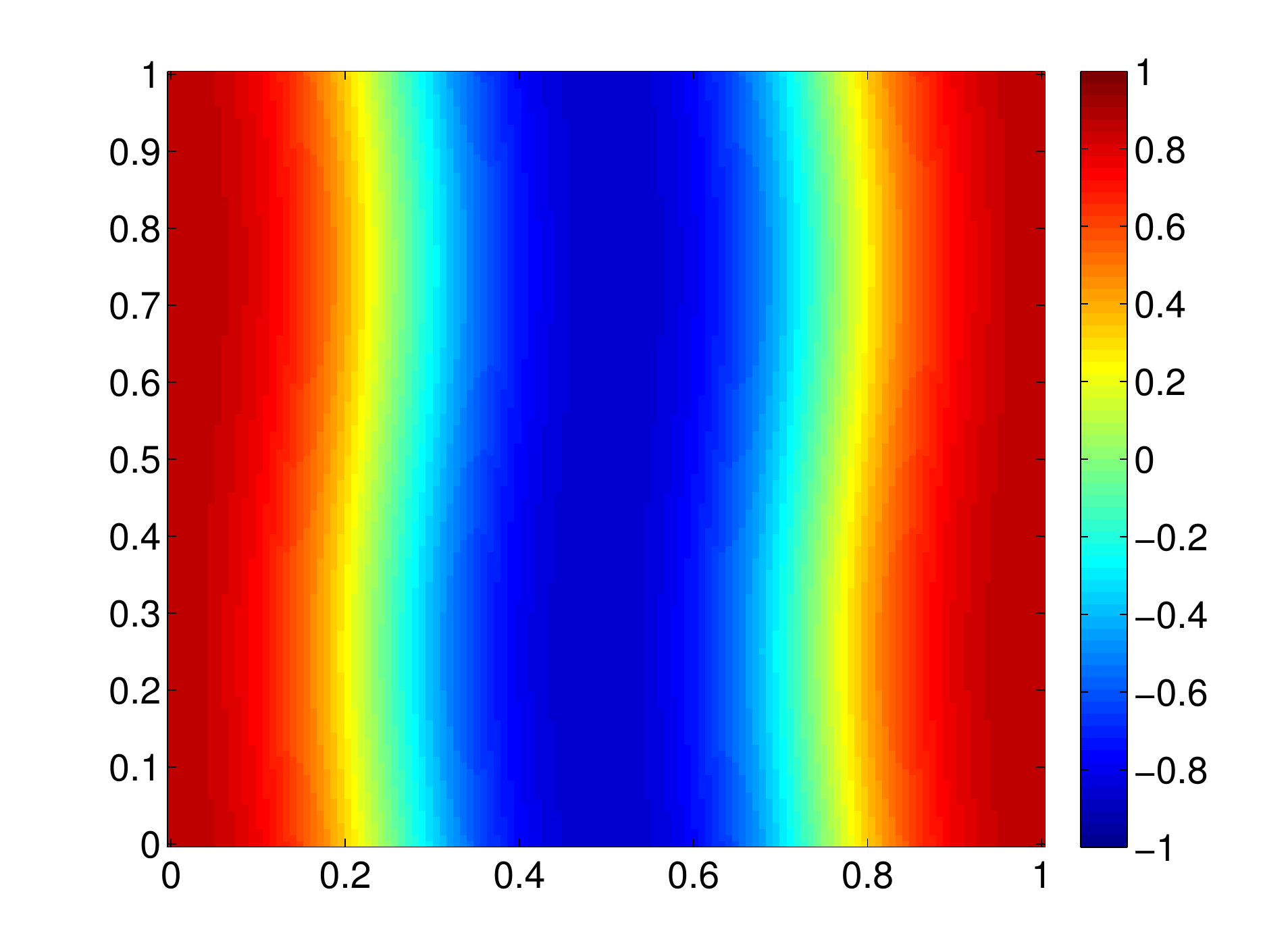}
\caption{$\gamma=0.005$}
\label{fig:1stshear0005}
\end{subfigure}
\begin{subfigure}[b]{0.18\textwidth}
\includegraphics[width=\textwidth]{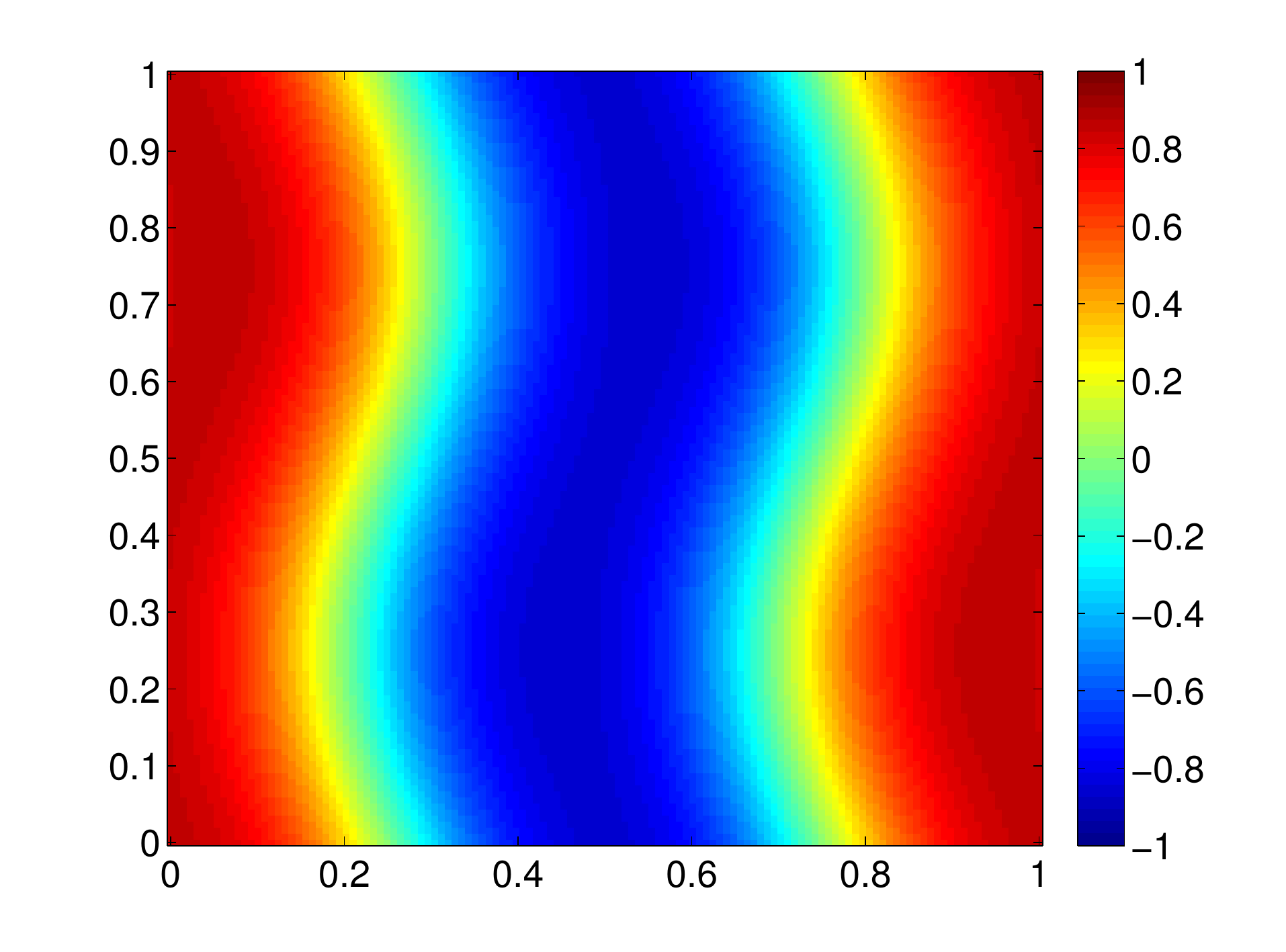}
\caption{$\gamma=0.02$}
\label{fig:1stshear002}
\end{subfigure}
\begin{subfigure}[b]{0.18\textwidth}
\includegraphics[width=\textwidth]{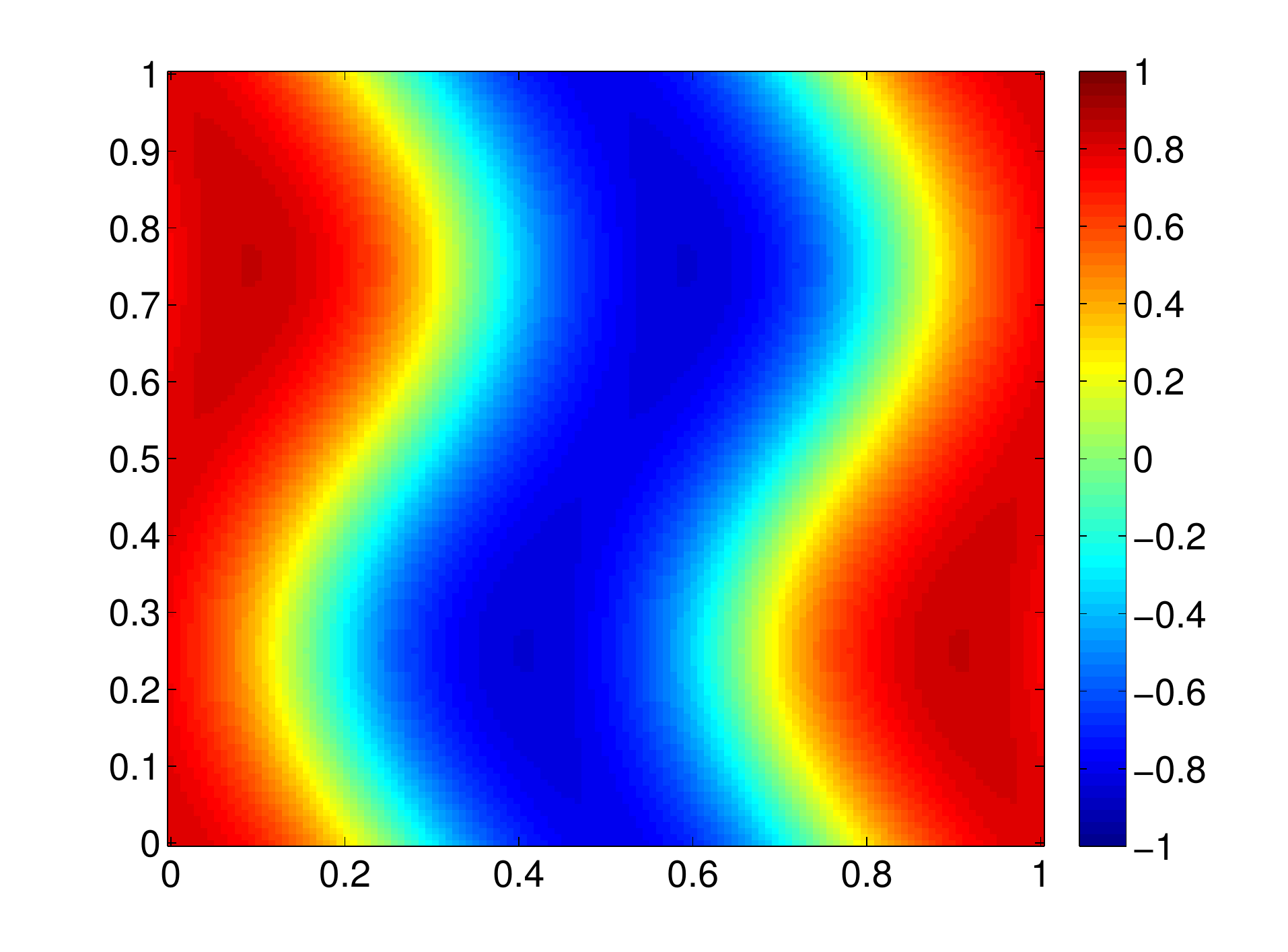}
\caption{$\gamma=0.035$}
\label{fig:1stshear0035}
\end{subfigure}
\begin{subfigure}[b]{0.18\textwidth}
\includegraphics[width=\textwidth]{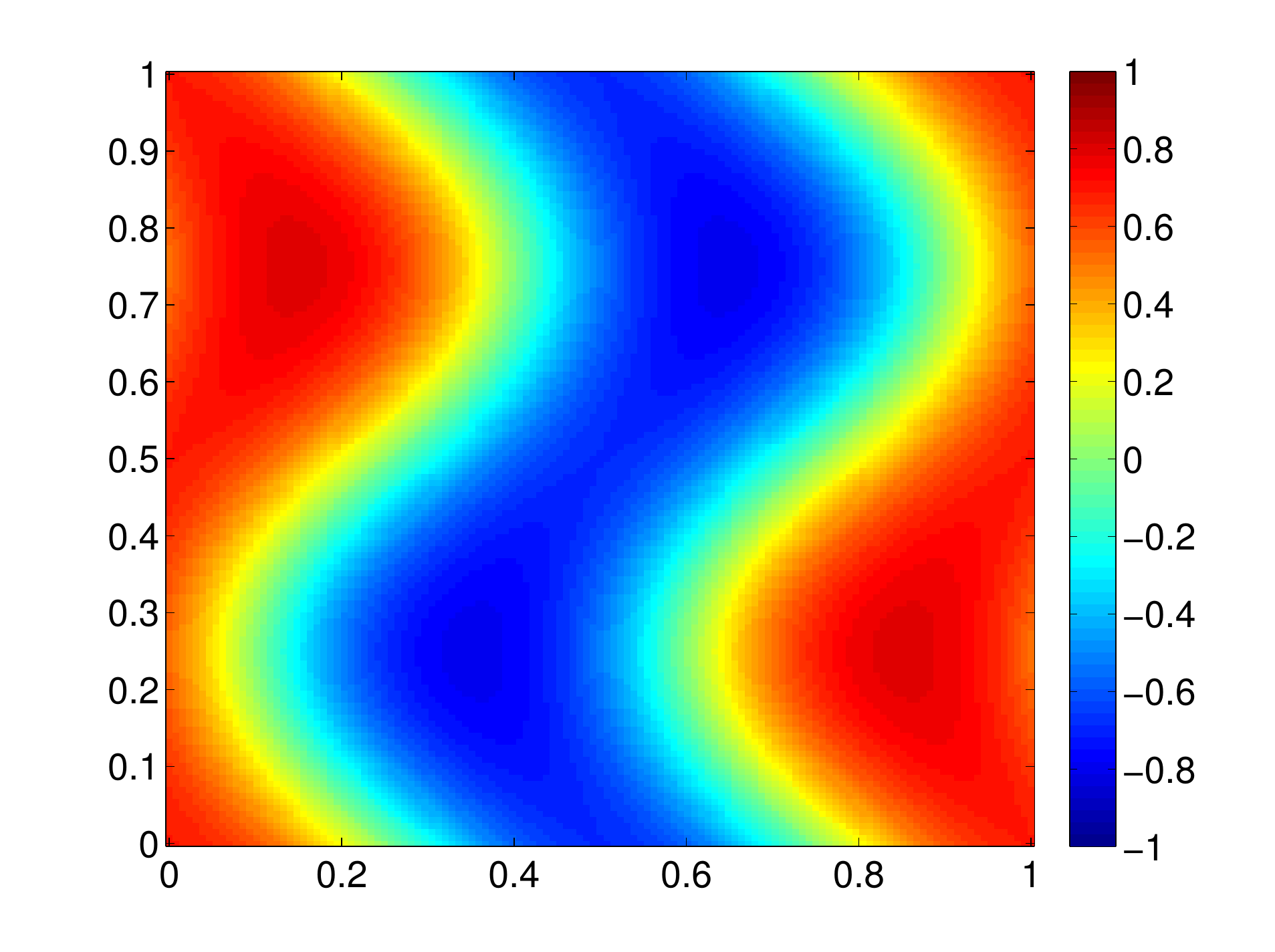}
\caption{$\gamma=0.05$}
\label{fig:1stshear005}
\end{subfigure}
\begin{subfigure}[b]{0.18\textwidth}
\includegraphics[width=\textwidth]{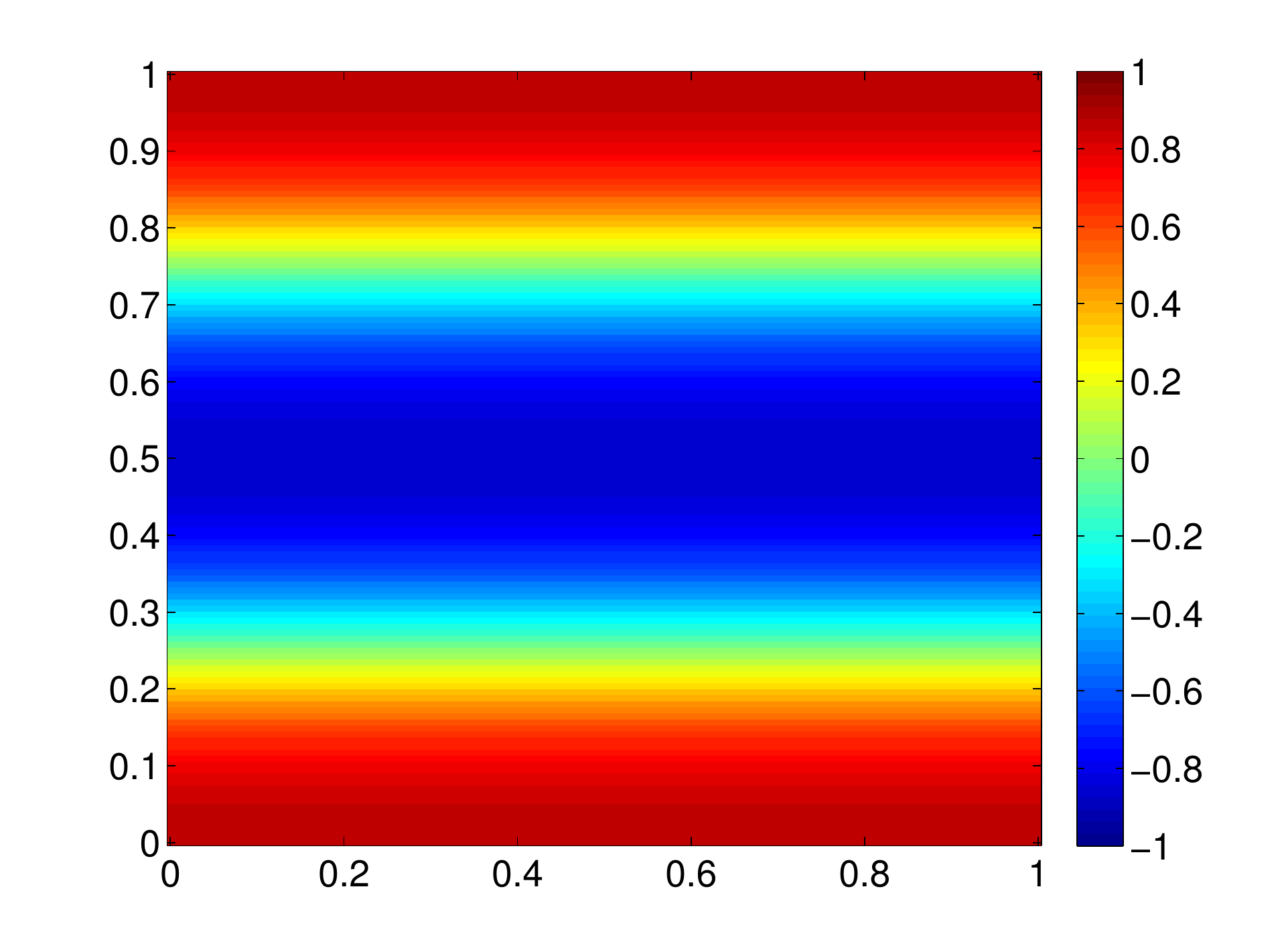}
\caption{$\gamma=0.065$}
\label{fig:1stshear0065}
\end{subfigure}
\begin{subfigure}[b]{0.18\textwidth}
\includegraphics[width=\textwidth]{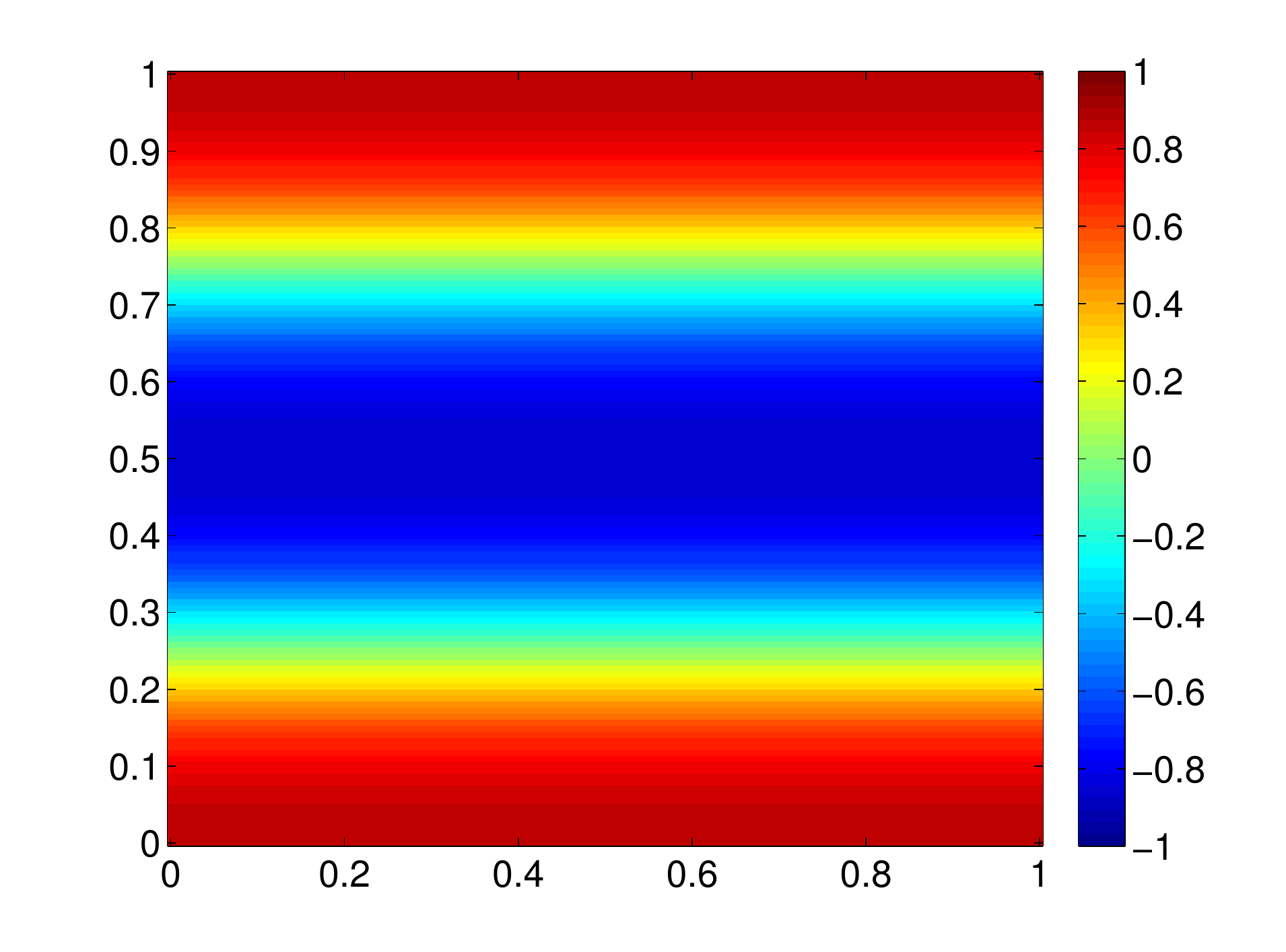}
\caption{$\gamma=0.08$}
\label{fig:1stshear008}
\end{subfigure}
\caption{Saddle points    for the model \eqref{GL_dyna}.
}
\label{fig:2D_saddle}
\end{center}
\end{figure}

   For the first case in equation \eqref{GL_dyna}, the shear    force exists only in the $x$ direction.
   As $\gamma$ increases, the sequence of the profiles
   of the saddle point is  shown in Figure \ref{fig:2D_saddle}. We   have the following interesting observations from this figure:
   the profiles of the saddle points
   get more and more stretched along the shear direction
   until a  lamellar phase
   is attained for   $\gamma$ large enough.
   In fact,   the lamellar phase shown in the last two subfigures   (Figure \ref{fig:1stshear0065} and \ref{fig:1stshear008}) is always a saddle point for
   any value of  $\gamma$.
   It seems to have a critical $\gamma_*$ between $0.05$   and $0.065$  such that
  for $\gamma<\gamma_*$, there are two saddles:
  one is   twisted   and the other is  lamellar,
  and for $\gamma>\gamma_*$, it seems only
  one index-1 saddle point, the lamellar phase.
   To determine which saddle point
   has the minimal action of escape a
   metastable state for a specific $\gamma<\gamma_*$,
   one needs to further run the minimum action method as in    \cite{MH2008}.

\begin{figure}[htbp]
\begin{center}
\begin{subfigure}[b]{0.18\textwidth}
\includegraphics[width=\textwidth]{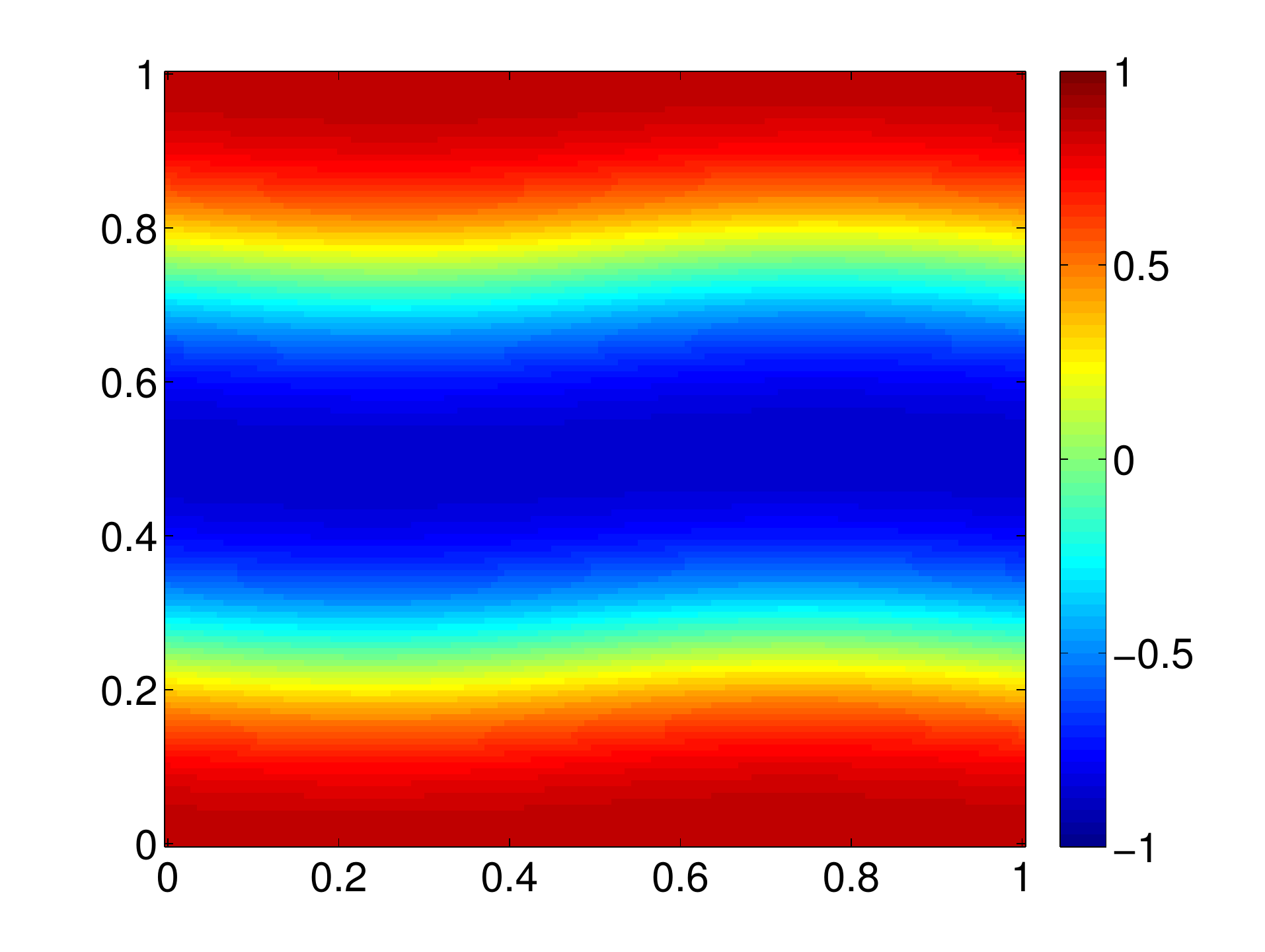}
\caption{$\gamma=0.005$}
\label{fig:2ndshear0005}
\end{subfigure}
\begin{subfigure}[b]{0.18\textwidth}
\includegraphics[width=\textwidth]{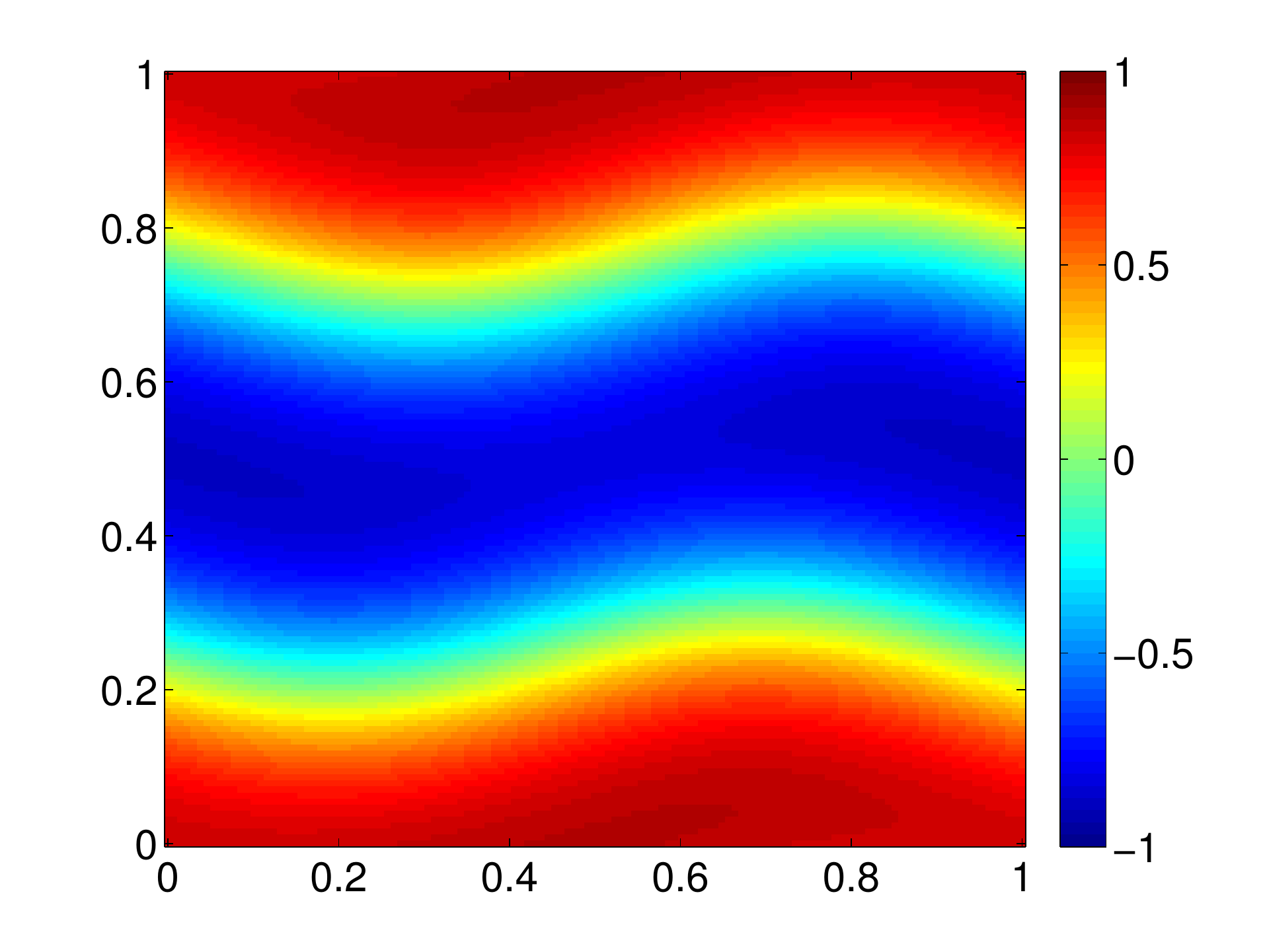}
\caption{$\gamma=0.02$}
\label{fig:2ndshear002}
\end{subfigure}
\begin{subfigure}[b]{0.18\textwidth}
\includegraphics[width=\textwidth]{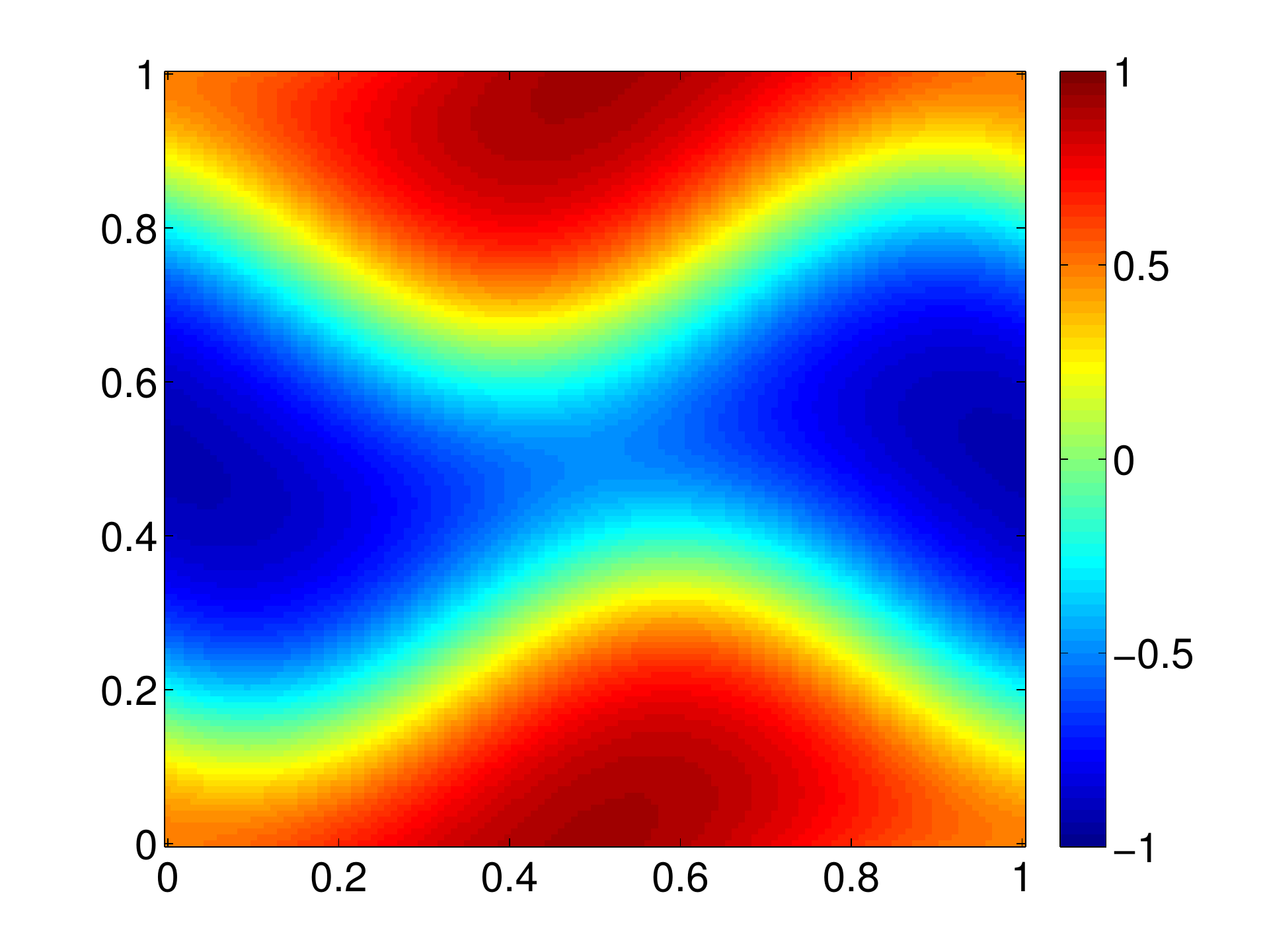}
\caption{$\gamma=0.05$}
\label{fig:2ndshear005}
\end{subfigure}
\begin{subfigure}[b]{0.18\textwidth}
\includegraphics[width=\textwidth]{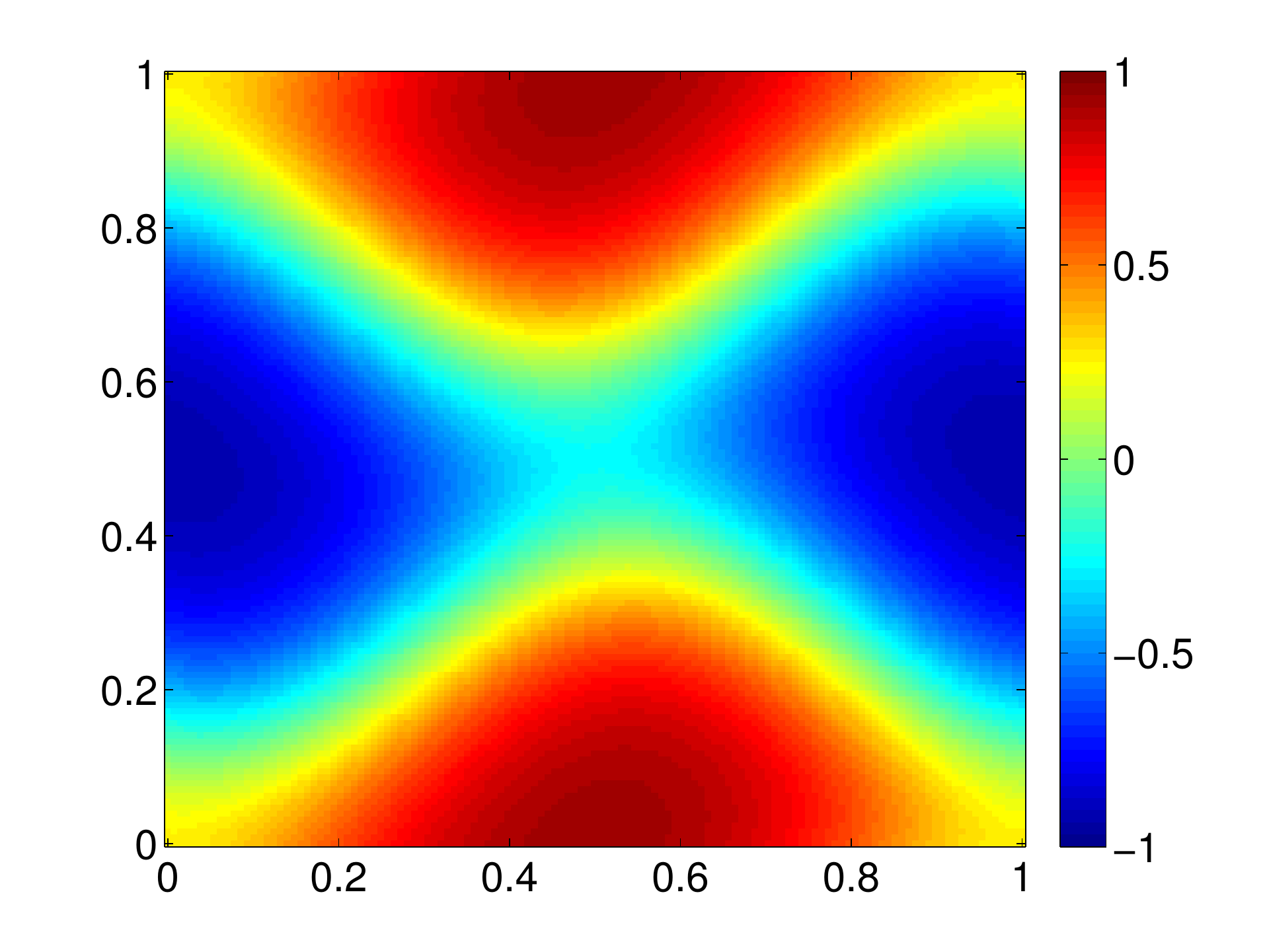}
\caption{$\gamma=0.0575$}
\label{fig:2ndshear00575}
\end{subfigure}
\begin{subfigure}[b]{0.18\textwidth}
\includegraphics[width=\textwidth]{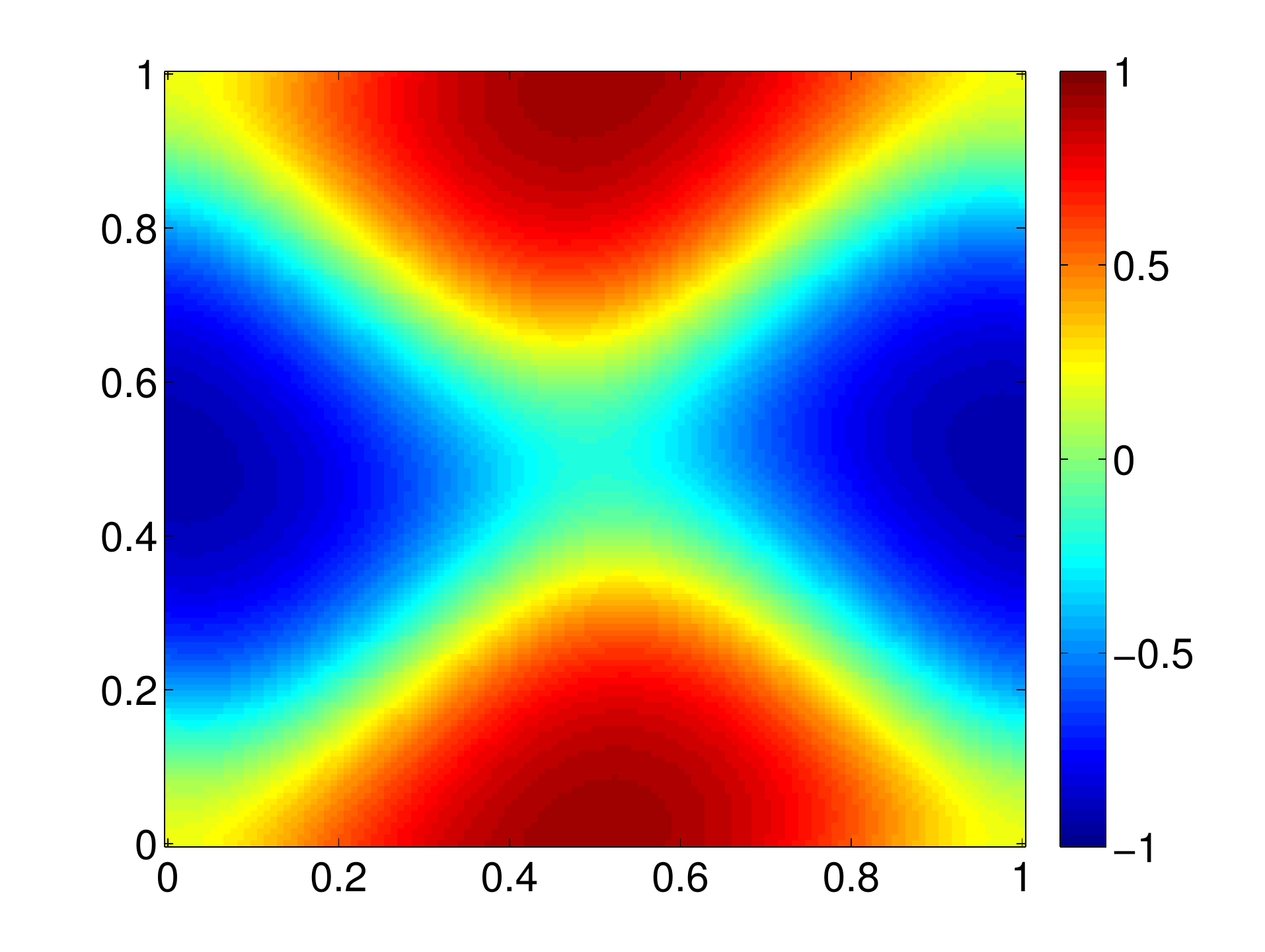}
\caption{$\gamma=0.0585$}
\label{fig:2ndshear00585}
\end{subfigure}
\begin{subfigure}[b]{0.18\textwidth}
\includegraphics[width=\textwidth]{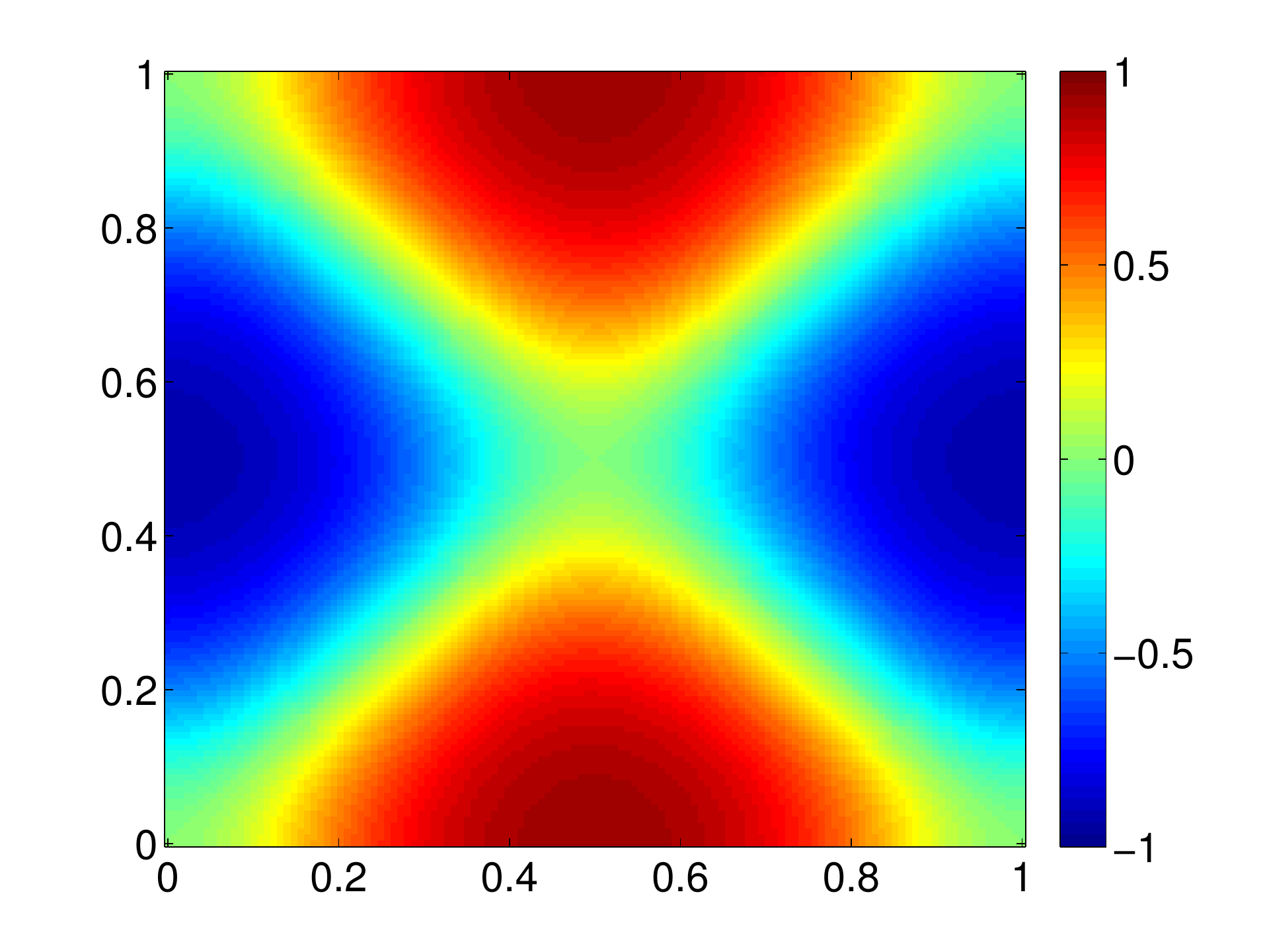}
\caption{$\gamma=0.1$}
\label{fig:2ndshear01}
\end{subfigure}
\caption{Saddle points  for the model \eqref{GL_dyna_new}.
}
\label{fig:2D_saddle_new}
\end{center}
\end{figure}

%

   For the second shear case in equation \eqref{GL_dyna_new}, the shear flow is no longer
   restricted in certain direction and is more general. In this case, the transition states with various shear rates are shown in Figure \ref{fig:2D_saddle_new}.
  The shear  ``twists''
  the profiles again but in different patterns.
 Similarly to the first case,  the saddle point is
 finally unchanged  when $\gamma$ is  sufficiently large. And eventually, the saddle point forms an ``X'' shape.
  But for small shear rate,
 the ``X''  shaped saddle point in Figure \ref{fig:2ndshear01}
 does not exit, unlike the lamellar phase in the previous shear case. Thus,
  it seems  to have only one  index-1 saddle point at any $\gamma$,
 except for the symmetric images.   
In summary, the shear acting on the Ginzburg-Landau energy landscape induces a variety of different patterns of the saddle points
and transition mechanisms. Our simplified GAD offers a useful tool for locating these saddle points with economic computational costs.

\section{Concluding Remarks}\label{sec_conclu}
    We present  a simplified  GAD for the non-gradient system in $\Real^d$
    to search   saddle points.
    It is a flow in $\Real^{2d}$ rather than in $\Real^{3d}$.
      Only one direction variable and one position variable are required in this new GAD.  So, it has the same computational cost as
  the  GAD   for the gradient system.
  Although we only show the result  for index-1 saddle points in this paper,
  it is not difficult to extend to index-$k$ saddle points by   following the
  approach in the original GAD paper \cite{GAD2011}.

    Our   numerical tests include the   Allen-Cahn equation in
  a periodic box with the presence of shear flow and we find the changes of saddle points  when the system is   subject to the various shear flows.
 In the end, we need to point out that although index-1 saddle points
 seem important for the  rare-event transitions in the non-gradient systems,  the saddle point found by the GAD may not
 directly be the true transition state. To quantify the minimal action, the minimum action method may still be  necessary to evaluate the action
 to the saddle points.
One may combine the MAM for the path and  the GAD for the final ending point on the path
to construct a hybrid method, in a similar style to
the climbing string method\cite{Ren2013} for the gradient system.
 
\bibliography{./simGAD}

\bibliographystyle{siam} 

\end{document}